\documentclass[12pt,reqno]{amsproc}
\usepackage{amsmath,graphicx,amssymb,amsthm}
\usepackage{rotating}

\usepackage[left]{lineno}
\usepackage{mdframed}

\theoremstyle{plain}
\def\B{\mathcal B}

\def\B{\mathcal B}

\def\H{\mathcal H}

\def\M{\mathcal M}

\def\P{\mathcal P}
\def\R{\mathcal R}

\def\amslatex{$\mathcal{A}\kern-.1667em\lower.5ex\hbox{$\mathcal{M}$}\kern-.125em\mathcal{S}$-\LaTeX}

\newtheorem{set}{set}[section]

\newtheorem{Lemma}[set]{Lemma}

\newtheorem{Theorem}[set]{Theorem}
\newcommand{\define}{\mathrel{\hbox{$\equiv$\hskip -.90em \lower .47ex \hbox{$\leftharpoondown$}}}}
\newcommand{\enifed}{\mathrel{\hbox{$\equiv$\hskip -.90em \lower .47ex \hbox{$\rightharpoondown$}}}}

\numberwithin{equation}{section}
\makeatletter
\newcommand{\LeftEqNo}{\let\veqno\@@leqno}

\makeatother

\begin{document}

\

\vspace{-2cm}


\title{On irreducible operators in factor von Neumann algebras}

\author{Junsheng Fang}
\address{School of Mathematical Sciences, Dalian University of
Technology, Dalian, 116024, China}
\curraddr{College of Mathematics and Information Science, Hebei Normal University, Shijiazhuang, 050024, China}
\email{junshengfang@hotmail.com}
\thanks{Junsheng Fang was partly supported by NSFC(Grant No.11431011) and a start up funding from Hebei Normal University.}

\author{Rui Shi}
\address{School of Mathematical Sciences, Dalian University of
Technology, Dalian, 116024, China}
\email{ruishi@dlut.edu.cn, ruishi.math@gmail.com}
\thanks{Rui Shi was partly supported by NSFC(Grant No.11401071) and the Fundamental Research Funds for the Central Universities (Grant No.DUT18LK23).}

\author{Shilin Wen}
\address{School of Mathematical Sciences, Dalian University of
Technology, Dalian, 116024, China}
\email{shilinwen127@gmail.com}
\thanks{}





\begin{abstract}
Let $\M$ be a factor von Neumann algebra with separable predual and let $T\in \M$. We call $T$ an irreducible operator (relative to $\M$) if $W^*(T)$
is an irreducible subfactor of $\M$, i.e., $W^*(T)'\cap \M={\mathbb C} I$. In this note, we show that the set of irreducible operators in $\M$ is a dense $G_\delta$ subset of
$\M$ in the operator norm. This is a natural generalization of a theorem of Halmos.
\end{abstract}

\maketitle

\section{Introduction}
In~\cite{Hal}, Halmos proved the following theorem. Let $\H$ be a separable (finite or infinite-dimensional) complex Hilbert space. Then the set of irreducible operators on $\H$ is a dense $G_\delta$ subset of $\B(\H)$ in the operator norm. Recall that an operator $T\in \B(\H)$ is irreducible if $T$ has no nontrivial reducing subspaces, i.e., if $P$ is a projection such that $PT=TP$ then $P=0$ or $P=I$. We refer to~\cite{Rad} for a beautiful short proof. In this note, we generalize the above theorem to arbitrary factor von Neumann algebras with separable predual. Let $\M$ be a factor von Neumann algebra with separable predual and let $T\in \M$. We call $T$ an irreducible operator (relative to $\M$) if $W^*(T)$
is an irreducible subfactor of $\M$, i.e., $W^*(T)'\cap \M={\mathbb C} I$. We show that the set of irreducible operators in $\M$ is a dense $G_\delta$ subset of
$\M$ in the operator norm.

Let $A \in \mathcal{B}(\mathcal{H})$ and $B \in \mathcal{B}(\mathcal{K})$, where $\mathcal{H}, \mathcal{K}$ are  Hilbert spaces. For every operator $X \in \mathcal{B}(\mathcal{K}, \mathcal{H})$, we define an operator $$\tau_{A,B}(X)= AX - XB.$$
	$\tau_{A,B}$ is called a Rosenblum operator~\cite{Ros,Rad2}.

\begin{Lemma}[Corollary 0.13 of~\cite{Rad2}]\label{Rosenblum}
If $\sigma(A)\cap \sigma(B)=\emptyset$, then $AX=XB$ implies $X=0$.
\end{Lemma}

Another ingredient in the proof of our main result is related to the generator problem of factor von Neumann algebras with separable predual. Precisely, we need the following lemma.
\begin{Lemma}\label{Popa}
Let $\mathcal{M}$ be a factor von Neumann algebra with separable predual. Then there exists a singly generated irreducible subfactor $\mathcal{N}$ in $\mathcal{M}$.
\end{Lemma}
\begin{proof}
It is well-known that if $\M$ is type ${\rm I}$, ${\rm II}_{\infty}$ or ${\rm III}$, then $\M$ is singly generated. So the lemma is clear in these cases. When $\M$ is a type ${\rm II}_1$ factor with separable predual, by~\cite{Pop}, there exists a hyperfinite irreducible subfactor in $\M$ which is singly generated.
\end{proof}

\section{Main result}
\begin{Theorem}
Let $\M$ be a factor von Neumann algebra with separable predual. Then the set of irreducible operators in $\M$ is a dense $G_\delta$ subset of
$\M$ in the operator norm.
\end{Theorem}
\begin{proof}
Let $T\in \M$ and $\epsilon>0$. We need to show that there exists an operator $S\in \M$ such that $\|T-S\|<\epsilon$ and $S$ is irreducible relative to $\M$, i.e., if $P\in \M$ is a projection such that $PS=SP$, then $P=0$ or $P=I$.

Write $T=A+iB$, where both $A$ and $B$ in $\M$ are self-adjoint operators. By the spectral theorem for self-adjoint operators, there exist $\lambda_1<\lambda_2<\cdots<\lambda_n\in\sigma(A)$ and projections $E_1,E_2,\ldots, E_n\in \M$ such that $\sum_{j=1}^n E_j=I$ and $\|A-\sum_{j=1}^n\lambda_jE_j\|<\frac{1}{4}\epsilon$.  Let $A_1=\sum_{i=1}^n\lambda_iE_i$ and $B_{ij}=E_iBE_j$. By the spectral theorem for self-adjoint operators again, there exist $\eta_{i1},\eta_{i2},\ldots,\eta_{im_i}\in\sigma(B_{ii})$ and projections $F_{i1},F_{i2},\ldots, F_{im_i}\in \M$ such that $\sum_{j=1}^{m_i} F_{ij}=E_i$ and $\|B_{ii}-\sum_{j=1}^{m_i}\eta_{ij}F_{ij}\|<\frac{1}{4}\epsilon$. Let $B_{ii}'=\sum_{j=1}^{m_i}\eta_{ij}F_{ij}$.

Define $T_1=A_1+iB_1$, where $B_1$ is self-adjoint, defined in the form
\[  B_{1}=
  \bordermatrix	
  {&E_1&E_2&\cdots&E_n \cr
  	E_1 & B'_{11} & B_{12} & \cdots & B_{1n} \cr
  	E_2 & B_{21}& B'_{22} & \cdots & B_{2n} \cr
  	\vdots & 	\vdots & \vdots & \ddots & \vdots\cr
  	E_n & B_{n1} & B_{n2} & \cdots & B'_{nn} \cr}.
  \]
Then $T_1$ can be expressed in the form
\[ T_{1} = A_{1} + i B_1 =
	\bordermatrix	
	{&E_1&E_2&\cdots&E_n \cr
		E_1 & \lambda_1 & 0 & \cdots & 0 \cr
		E_2 & 0 & \lambda_2 & \cdots & 0 \cr
		\vdots & 	\vdots & \vdots & \ddots & \vdots\cr
		E_n & 0 & 0 & 	\cdots & \lambda_n \cr}
	+
	\bordermatrix	
	{&E_1&E_2&\cdots&E_n \cr
		E_1 & B_{11}' & B_{12} & \cdots & B_{1n} \cr
		E_2 & B_{21}& B_{22}' & \cdots & B_{2n} \cr
		\vdots & 	\vdots & \vdots & \ddots & \vdots\cr
		E_n & B_{n1} & B_{n2} & \cdots & B_{nn}' \cr}.
	\]
Note that

\begin{align}\label{T_1}
\|T-T_1\|=\|(A+iB)-(A_1+iB_1)\|\leq \|A-A_1\|+\|B-B_1\|<\frac{1}{2}\epsilon.
\end{align}

For $1\leq i\leq n$, we can choose real numbers $\lambda_{i1},\lambda_{i2},\ldots,\lambda_{im_{i}}$ such that
\begin{enumerate}
	\item the inequality $\|\lambda_iE_i-\sum_{j=1}^{m_i}\lambda_{ij}F_{ij}\|<\frac{1}{8}\epsilon$ holds for every $i$;
	\item $\lambda_{11}<\lambda_{12}<\cdots<\lambda_{1m_1}<\lambda_{21}<\cdots<\lambda_{2m_2}<\cdots<\lambda_{n1}<\cdots<\lambda_{nm_n}$.
\end{enumerate}

Define $A_2=\sum_{i=1}^n\sum_{j=1}^{m_i}\lambda_{ij}F_{ij}$. Then $\|A_1-A_2\|<\frac{1}{8}\epsilon$.  Now we make a small self-adjoint perturbation $B_2$ of $B_1$ such that each off-diagonal entry of $B_2$, with respect to the decomposition
\[
I=\sum_{i=1}^n\sum_{j=1}^{m_i}F_{ij},
\]
is nonzero. That is we can construct a self-adjoint operator $B_2$ in $\M$ such that $\|B_2-B_1\|<\frac{1}{8}\epsilon$ and $F_{ij}B_2F_{i'j'}\neq 0$ if $i\neq i'$ or $j\neq j'$. Let $T_2$ be defined in the form $T_2=A_2+iB_2$, for
{\begin{equation}
A_{2}=
	\bordermatrix	
	{&F_{11}&F_{12}&\cdots&F_{nm_n} \cr
		F_{11} & \lambda_{11} & 0 & \cdots & 0 \cr
		F_{12} & 0 & \lambda_{12} & \cdots & 0 \cr
		\vdots & 	\vdots & \vdots & \ddots & \vdots\cr
		F_{nm_n} & 0 & 0 & 	\cdots & \lambda_{nm_n} \cr}\quad
	\mbox{and}\quad
	B_{2}=\bordermatrix	
	{&F_{11}&F_{12}&\cdots&F_{nm_n} \cr
		F_{11} & \eta_{11} & \ast & \cdots & \ast \cr
		F_{12} & \ast & \eta_{12} & \cdots & \ast \cr
		\vdots & 	\vdots & \vdots & \ddots & \vdots\cr
		F_{nm_n} & \ast & \ast & \cdots & \eta_{nm_n} \cr},\nonumber
\end{equation}}
where each $\ast$-entry is nonzero. By applying (\ref{T_1}), it follows that
\begin{equation}\label{T_2}
	\|T-T_2\|\leq \|T-T_1\|+\|T_1-T_2\|<\frac{3}{4}\epsilon.
\end{equation}

Since $\M$ is a separable factor, $F_{ij}\M F_{ij}$ is also a separable factor. By Lemma \ref{Popa}, we can find positive elements $X_{ij},Y_{ij}\in F_{ij}\M F_{ij}$ such that $\{X_{ij},Y_{ij}\}''$ is an irreducible subfactor of $F_{ij}\M F_{ij}$. Now we can choose $\delta>0$ sufficiently small such that the spectra of $\lambda_{ij}F_{ij}+\delta X_{ij}$ are pairwise disjoint, for $1\leq i\leq n$ and $1\leq j\leq m_i$.

Let $T_3$ be defined in the form $T_3=A_3+i B_3$, for
\begin{equation}
	A_3=\bordermatrix	
	{&F_{11}&F_{12}&\cdots&F_{nm_n} \cr
		F_{11} & \lambda_{11}+\delta X_{11} & 0 & \cdots & 0 \cr
		F_{12} & 0 & \lambda_{12}+\delta X_{12} & \cdots & 0 \cr
		\vdots & 	\vdots & \vdots & \ddots & \vdots\cr
		F_{nm_n} & 0 & 0 & 	\cdots & \lambda_{nm_n}+\delta X_{nm_n} \cr}
	\nonumber\
\end{equation} and
\begin{equation}
		B_3=\bordermatrix	
	{&F_{11}&F_{12}&\cdots&F_{nm_n} \cr
		F_{11} & \eta_{11}+\delta Y_{11} & \ast & \cdots & \ast \cr
		F_{12} & \ast & \eta_{12}+\delta Y_{12} & \cdots & \ast \cr
		\vdots & 	\vdots & \vdots & \ddots & \vdots\cr
		F_{nm_n} & \ast & \ast & \cdots & \eta_{nm_n}+\delta Y_{nm_n} \cr},\nonumber
	\end{equation}
where each $\ast$-entry is the same as in $T_2$.
Then, clearly, if $\delta>0$ is small enough, we have
\begin{equation}\label{T_3}
	\|T_2-T_3\|<\frac{1}{4}\epsilon.
\end{equation}
Hence, the inequalities (\ref{T_2}) and (\ref{T_3}) entail that $\|T-T_3\|<\epsilon$.

We assert that $T_3$ is irreducible relative to $\M$. Let $P\in \M$ be a projection commuting with $T_3$. Then $PA_3=A_3P$ and $PB_3=B_3P$. Write $P=(P_{ab})_{1\leq a,b\leq k}$ with respect to the decomposition $I=\sum_{i=1}^n\sum_{j=1}^{m_i}F_{ij}$, where $k=\sum_{i=1}^nm_{i}$. That $PA_3=A_3P$ implies that
\begin{equation}
	(\lambda_{11}F_{11}+\delta X_{11})P_{12}=P_{12}(\lambda_{12}F_{12}+\delta X_{12}).\nonumber
\end{equation}
Since $\sigma(\lambda_{11}F_{11}+\delta X_{11})\cap \sigma (\lambda_{12}F_{12}+\delta X_{12})=\emptyset$, we have $P_{12}=0$ by Lemma \ref{Rosenblum}. Similarly, we have $P_{ab}=0$ for $1\leq a\neq b\leq k$. Thus $P=\sum_{a=1}^k P_{aa}$ is diagonal with respect to the decomposition $I=\sum_{i=1}^n\sum_{j=1}^{m_i}F_{ij}$.

By the construction that $\{X_{ij}, Y_{ij}\}''$ is an irreducible subfactor of $F_{ij}\M F_{ij}$, it follows that $P_{aa}$ is either {{0}} or $I_{F_{ij}\M F_{ij}}$ for each $a$. Since the off-diagonal entries of $B_3$ are nonzero, an easy calculation shows that, if $P_{11}=0$, then $P_{aa}=0$ for $1\leq a\leq k$. Therefore, $P=0$ or $P=I$. This proves that $T_3$ is irreducible relative to $\M$.

The remainder is to prove the set of irreducible operators relative to $\M$ is a $G_\delta$ subset of $\M$ in the operator norm. The proof is similar to that provided by Halmos. For the sake of completeness, we include the details. Let $\P$  be the set of all those selfadjoint operators $P$ in $\M$ for which $0\leq P\leq I$. Let $\P_0$ be the subset of those elements of $\P$ that are \emph{not} scalar multiples of the identity. Since $\P$ is a weakly closed subset of the unit ball of $\M$, it is weakly compact, and hence the weak topology for $\P$ is metrizable. Since the set of scalars is weakly closed, it follows that $\P_0$ is weakly locally compact. Since the weak topology for $\P$ has a countable base, the same is true for $\P_0$, and therefore $\P_0$ is weakly $\sigma$-compact. Let $\P_1,\P_2,\ldots$ be weakly compact subsets of $\P_0$ such that $\cup_{n=1}^\infty \P_n=\P_0$.

It is to be proved that the set of reducible operators relative to $\M$, denoted by $\R(\M)$, is an $F_\sigma$ set in the operator norm topology. Let $\hat{\P}_n$ be the set of all those operators $T$ in $\M$ for which there exists a $P\in \P_n$ such that $TP=PT$. Then $\cup_{n=1}^\infty\hat{\P}_n=\R(\M)$.

The proof can be completed by showing that each $\hat{\P}_n$ is closed in the operator norm. Suppose that $T_k\in \hat{\P}_n$ and $\lim_{k\rightarrow \infty}T_k=T$ in the operator norm. For each $k$, find a $P_k\in \P_n$ such that $T_kP_k=P_kT_k$. Since $\P_n$ is weakly compact and metrizable, we may assume that $P_k$ is weakly convergent to $P$ in $\P_n$. Then $\lim_{k\rightarrow\infty}T_kP_k= TP$ and $\lim_{k\rightarrow\infty}P_kT_k=PT$ in the weak operator topology. Hence, $TP=PT$ and $T\in \hat{\P}_n$. This implies that $\hat{P}_n$ is closed in the operator norm and $\R(\M)$ is an $F_\sigma$ set in the operator norm.
\end{proof}

\end{document}